%% file: general-de-finetti.tex
\documentclass[12pt]{amsart}

\usepackage[margin={1.5in, 1in}]{geometry}
\usepackage{amsmath,amssymb,amsfonts}
\usepackage{graphicx}
\usepackage{textcomp}
\def\BibTeX{{\rm B\kern-.05em{\sc i\kern-.025em b}\kern-.08em
    T\kern-.1667em\lower.7ex\hbox{E}\kern-.125emX}}

\usepackage{microtype}

\usepackage[T1]{fontenc}
\usepackage[utf8]{inputenc}
\usepackage{mathscinet}
\input{headers.tex}

\input{defs.tex}

\input{preamble.tex}

\usepackage[hide=false,setmargin=true,marginparwidth=1.25in]{marginalia}

 \usepackage[colorlinks,citecolor=blue,urlcolor=blue,linkcolor=RawSienna]{hyperref}

\usepackage[capitalize]{cleveref}

\crefname{lemma}{Lemma}{Lemmas}
\crefname{corollary}{Corollary}{Corollaries}
\crefname{theorem}{Theorem}{Theorems}

\crefformat{enumi}{(#2#1#3)}

\crefname{problem}{Conjecture}{Conjectures}

\crefname{statement}{Statement}{Statement}

\usepackage[foot]{amsaddr}

\begin{document}
\setlength{\marginparsep}{0.05in}

\title[de Finetti's theorem, regular versions, and strong laws]
{de Finetti's theorem and the existence of regular conditional distributions and strong laws on exchangeable 
algebras
}

\author[P.\ Potaptchik]{Peter Potaptchik$^{1}$}

\address{$^1$ University of Oxford}
\author[D.\ M.\ Roy]{Daniel M. Roy$^{2,3}$}
\address{$^2$ University of Toronto,  
$^3$ Vector Institute}
\author[D.\ Schrittesser]{David Schrittesser$^{4}$}
\address{$^4$ Institute for Advanced Studies in Mathematics, Harbin Institute for Technolgy}

\maketitle

\begin{abstract}
We show the following generalizations of the de Finetti--Hewitt--Savage theorem:
Given an exchangeable sequence of random elements, the sequence is conditionally i.i.d.\ if and only if each random element 
admits a regular conditional distribution given the exchangeable $\sigma$-algebra (equivalently, the shift invariant or the tail algebra).
We use this result, which holds without any regularity or technical conditions, to demonstrate that 
any exchangeable sequence of random elements whose common distribution is Radon is conditional iid.
\end{abstract}

\section{Introduction}

In this paper, we prove a generalization of de Finetti's representation theorem for exchangeable sequences.
According to this theorem, any infinite, exchangeable sequence of real random variables is identically and independently distributed, conditioned on the exchangeable algebra.
Besides being one of the simplest examples of a probabilistic symmetry and invariance principle,
de Finetti's theorem is considered a central result in Bayesian statistics, where it leads 
from a (subjective) distribution on a sequence of observables to the existence of a (random) parameter, conditioned on which the sequence is conditionally i.i.d.
See \cite[p.~12]{schervish} for a discussion

Recall that a sequence of random elements $X = (X_i)_{i \in \Nats}$ is \inlinedfn{exchangeable} if for every $n \in \Nats$ and every permutation $\sigma$ of $\{1,\hdots, n\}$, the distributions of $(X_1,  \hdots, X_n)$ and $(X_{\sigma(1)}, \hdots, X_{\sigma(n)})$
are equal.

\begin{theorem}[de Finetti, Ryll-Nardzewski]\label{t.classic}
Let $X = (X_i)_{i \in \Nats}$ be a sequence of real random variables.
Then the following are equivalent:
\begin{enumerate}
\item\label{i.exchangeable} $X$ is exchangeable,
\item\label{i.conditionally.iid} $X$ is conditionally iid,
\item\label{i.mixed.iid} $X$ is mixed iid.
\end{enumerate}
\end{theorem}

The first form of the theorem, proven by de Finetti in 1928, established the equivalence of 
\cref{i.exchangeable} and \cref{i.mixed.iid}
for exchangeable sequences $X = (X_i)_{i \in \Nats}$ of \emph{binary} random variables (i.e., $X_i \in \{0,1\}$ a.s.; see \cite{definetti29}); de Finetti latter proved the same equivalence for real random variables \cite{definetti37}.
The equivalence between 
\cref{i.exchangeable} and \cref{i.conditionally.iid} 
was shown by Ryll-Nardzewski \citep{ryll-nardzewski}.

Given its foundational importance, over the years, much interest has been given to generalizations of this theorem to sequences of random elements 
in more general spaces.

\cref{t.classic} holds more generally for sequences of random elements 
$(X_i)_{i \in \Nats}$ taking values in any 
\inlinedfn{standard Borel space}, 
i.e., a measurable space $\big(S,\mathcal S\big)$ such that
$S$ is isomorphic to %
an uncountable Borel subset of a Polish space $S'$, equipped with the  restriction of the Borel algebra of $S'$,
via a map $f$ such that both $f$ and $f^{-1}$ are measurable.
This generalization is 
immediate using the isomorphism theorem for measures (see \cite[12.B, 17.43]{kechris}),
according to which any standard Borel space $S$ is isomorphic to $[0,1]$ 
via $f$ as above.\footnote{This observation is made, e.g., by Diaconis and Freedman \citep{diaconis-freedman}.}

A famous, further generalization was proved by Hewitt--Savage \cite{hewitt-savage}; namely, they allow the $X_i$ to take values in any compact Hausdorff space, equipped with its Baire $\sigma$-algebra.
Varadarajan \cite[p.~219]{varadarajan} shows that this also implies \cref{t.classic} for $X_i$ which take values in an \emph{analytic space} (the analytic spaces are the continuous images of Polish spaces).

That some restriction on the class of processes is needed was made clear by Dubins and Freedman, who gave a counterexample: 
There is a separable metric space $S$ for which \cref{t.classic} fails, that is, there is a sequence of random elements of this space that is exchangeable, but not mixed iid.
Later, Freedman \cite{freedman} also produced an example of an exchangeable sequence of random elements that is mixed iid, but not conditional iid.

Another approach, which has come to fruition more recently, is to put requirements directly on the distribution of $X$, instead of on the space in which the $X_i$ take their values.
In this direction, \citet[459H]{fremlin} shows a version of \cref{t.classic} in which it is supposed that $X$ be \emph{jointly quasi-Radon} distributed (for this definition, see \cite[411H]{fremlin}) as well as that $\distribution(X_i)$ be Radon, thereby verifying a conjecture made by Ressel \cite{ressel}.

Recall  that a \inlinedfn{Radon} probability measure is 
 a Borel probability measure $\mu$ on a Hausdorff topological space $S$ such that
 \begin{equation*}
\forall B \in \Borel(S)\; \mu(B) = \sup \{ \mu(K) \setdef\text{ $K \subseteq B$, $K$ is compact}\}
 \end{equation*}
 i.e., $\mu$ is inner regular.\footnote{Note that Radon probability measures, as per the definition given by Fremlin in \cite{fremlin}, are exactly the completions of our notion of Radon probability measures.}
In a recent breakthrough,  \citet{alam} used a nonstandard analytic argument to establish the following:
\begin{theorem}[{\cite[Thm.~4.7]{alam}}]\label{t.alam}
Let $X = (X_i)_{i \in \Nats}$ be an exchangeable sequence of random elements whose common distribution
$\distribution(X_i)$ is Radon.
Then $X$ is mixed iid.
\end{theorem}
\citet{alam} also shows that the Hewitt--Savage generalization can be derived as a corollary to the above theorem (or even less general theorems which require Radon-ness, such as Ressel's earlier result; %
see \cite[Appendix B]{alam}).

In the present paper, we provide an exact characterization of which exchangeable processes are conditionally iid.
Towards stating this result, fix a basic probability space $(\Omega, \Sigma, \PR)$ on which a sequence of random elements
is defined.
In the following theorem, 
$\mathcal{I}_{\infty}$ denotes the left-shift invariant $\sigma$-algebra, $\mathcal{T}_{\infty}$ denotes the tail $\sigma$-algebra, and $\mathcal{E}_{\infty}$
denotes the exchangeable $\sigma$-algebra with respect to this sequence. 
By a \inlinedfn{regular conditional distribution}
 of a random variable $X : (\Omega,\Sigma) \to (S,\mathcal S)$ given a  $\sigma$-algebra $\mathcal{G}$,
 we mean a Markov kernel 
 $\kappa : \Omega \times \mathcal S \to [0,1]$
 from $(\Omega, \mathcal G)$ to $(S,\mathcal S)$ such that, for all $B \in \mathcal{S}$ and $\PR$-almost all $\omega \in \Omega$, 
 we have $\kappa(\omega, B) = \Ex[\indicate\{X \in B\}|\mathcal{G}](\omega)$.

\begin{theorem}
    \label{thm:condiiddefinetti}
    Let $(\Omega, \Sigma, \PR)$ be a probability space and, on this space, let $(X_n)_{n \in \N}$ be a sequence of exchangeable random elements in a measurable space $(S, \mathcal{S})$. Let $\Sigma_{\infty}$ be one of $\mathcal{I}_{\infty}, \mathcal{T}_{\infty}, \mathcal{E}_{\infty}$. The following are equivalent: 
    \begin{enumerate}
        \item $X_1$ given $\Sigma_{\infty}$ admits a regular conditional distribution. \label{item:DFcond1}
        \item $(X_n)_{n \in \N}$ is conditionally iid given some sub $\sigma$-algebra of $\Sigma$.
        \label{item:DFcond2}
    \end{enumerate}
\end{theorem}

Note that in the above theorem, we do not place any restrictions whatsoever on either the process or its state space.
As an application, we are able to prove a strengthening of \cref{t.alam}:

\begin{theorem}\label{t.sufficient}
Let $X = (X_i)_{i \in \Nats}$ be an exchangeable sequence of random elements whose common distribution
$\distribution(X_i)$ is Radon.
Then $X$ is conditionally iid.
\end{theorem}
As we have already noted above,
at this level generality,
mixed iid does not entail conditional iid, as the counterexample due to Freedman \cite{freedman} shows.

In the above theorem, one might replace the hypothesis that
$\distribution(X_1)$ is 
Radon by the seemingly weaker hypothesis that
$\distribution(X_1)$ is tight
and outer regular on compact sets (see \cref{s.when.radon} for details). This does not, however, make the theorem stronger, since we show in \cref{p.when.radon} that
tightness and outer regularity on compact sets already imply Radon-ness of a probability measure.
Note, moreover, that if $S$ is a metric space, it is known that tightness of $\mu$ alone already implies that $\mu$ is Radon.

One of the reasons we are able to derive this strengthening of \cref{t.alam} from \cite{alam} 
is the generality of \cref{thm:condiiddefinetti}, providing a clear target (the conditional regular distribution of $X_1$) and allowing us to rephrase the question
of the existence of the underlying random measure as a matter of convergence of the sequence of empirical distributions.
Moreover, we show that Radon-ness already implies that the sequence of empirical distributions is almost surely tight.

\section{Conditionally IID De Finetti}
In this section, we provide necessary and sufficient conditions for the conditionally iid version of de Finetti's theorem. We begin by stating two versions of the strong law of large  numbers which will be instrumental in the proof. Both can be proved using elementary martingale methods.

Firstly, we need the following version of the strong law of large numbers for exchangeable sequences.
This result is standard for the exchangeable and tail $\sigma$-algebras; see, e.g., \citealt[][Theorem 12.17]{KlenkeAchim2020PTAC}. 
The result for the left-shift invariant $\sigma$-algebra follows since $\mathcal{I}_{\infty} \subseteq \mathcal{E}_{\infty}$, 
$$\lim_{n \to \infty}\frac{X_1 + \dots + X_n}{n} = \Ex[X_1|\mathcal{E}_{\infty}] \quad \text{a.s.}$$
and the left side is $\mathcal{I}_{\infty}$-measurable. 
 
\begin{theorem}\label{prop:exchslln}
    Let $(\Omega, \Sigma, \PR)$ be a probability space and $(X_n)_{n \in \N}$ an exchangeable sequence in $\mathcal{L}^1(\Omega, \Sigma, \PR)$. Let $\Sigma_{\infty}$ be one of $\mathcal{I}_{\infty}, \mathcal{T}_{\infty}, \mathcal{E}_{\infty}$. Then $$\frac{X_1 + \dots + X_n}{n} \to \Ex[X_1|\Sigma_{\infty}] \quad \text{a.s.}$$
\end{theorem}

Secondly, we need a version of the strong law of large numbers for conditionally iid sequences.
For a proof, see, e.g., \citealt[][Theorem~7]{PrakasaRaoB.L.S.2009Cicm}.

\begin{theorem}\label{prop:condiidslln}
    Let $(\Omega, \Sigma, \PR)$ be a probability space and $(X_n)_{n \in \N}$ a sequence in $\mathcal{L}^1(\Omega, \Sigma, \PR)$ that is bounded by some $K \in [0, \infty)$ and that is conditionally iid given $\mathcal{G} \subseteq \Sigma$ a sub $\sigma$-algebra. Then 
    $$\frac{X_1 + \dots + X_n}{n} \to \Ex[X_1|\mathcal{G}] \quad \text{a.s.}$$
\end{theorem}
With these two theorems at our disposal, we can now begin with the proof of our main theorem.

\begin{proof}[Proof of \cref{thm:condiiddefinetti}.]
    First, we show that \cref{item:DFcond1} implies \cref{item:DFcond2}. Let $\kappa$ be a regular conditional distribution of $X_1$ given $\Sigma_{\infty}$. In particular, $\kappa$ is a Markov kernel from $(\Omega, \Sigma_{\infty})$ to $(S, \mathcal{S})$. For any $A \in \mathcal{S}$, $\left(\indicate\{X_n \in A\}\right)_{n \in \N}$ is an exchangeable sequence in $\mathcal{L}^1(\Omega, \Sigma, \PR)$. Therefore, \cref{prop:exchslln} implies that, for $\PR$-almost all $\omega \in \Omega$,
\begin{equation} \label{eqn:DFdef}
\begin{split}
    \lim_{n \to \infty} \frac{\indicate\{X_1(\omega) \in A\} + \dots + \indicate\{X_n(\omega) \in A\}}{n} &=
     \Ex[\indicate\{X_1 \in A\}|\Sigma_{\infty}](\omega) 
     \\ &= \kappa(\omega, A).
    \end{split}
\end{equation}
Let $\mathcal{G}$ be any of $\mathcal{I}_{\infty}, \mathcal{T}_{\infty}, \mathcal{E}_{\infty}$ in $(2)$. We claim that $\kappa^{\infty}$ is a regular conditional distribution of $X := (X_n)_{n \in \N}$ given $\mathcal{G}$. 
For every $B \in \mathcal{S}^{\infty}$,
it suffices to show that, for $\PR$-almost all $\omega \in \Omega$,
$    \kappa^{\infty}(\omega, B) = \Ex[\indicate\{X \in B\}|\mathcal{E}_{\infty}](\omega).
$
Indeed, it is sufficient to show 
$$\int_E \kappa^{\infty}(\omega,B) \;\PR(d\omega) = \int_{E}\indicate\{X(\omega) \in B\}\;\PR(d\omega) $$
for $B = A_1 \times \dots \times A_m \times \mathcal{S}^{\infty}$ where $m \in \N$, $A_1, \dots, A_m \in \mathcal{S}$, and $E \in \mathcal{G}$. We have
\begin{align}
    &\int_{E}\kappa^{\infty}(\omega,B) \;\PR(d\omega) %
    = \int_{E}\prod_{i = 1}^m\kappa(\omega, A_i) \;\PR(d\omega) \notag \displaybreak[0]\\
    &= \int_{E}\prod_{i = 1}^m \lim_{n \to \infty} \frac{\indicate\{X_1(\omega) \in A_i\} + \dots + \indicate\{X_n(\omega) \in A_i\}}{n}\;\PR(d\omega) \label{eqn:DF1} \notag \displaybreak[0]\\
     &= \int_{E}\lim_{n \to \infty}\sum_{j_1, \dots , j_m = 1}^n  \prod_{i = 1}^m\frac{\indicate\{X_{j_i}(\omega) \in A_i\}}{n}\;\PR(d\omega) \displaybreak[0]\\
     &= \lim_{n \to \infty}\sum_{j_1, \dots , j_m = 1}^n\int_{E}\prod_{i = 1}^m\frac{\indicate\{X_{j_i}(\omega) \in A_i\}}{n}\;\PR(d\omega) \label{eqn:DF2}\displaybreak[0]\\
     &= \lim_{n \to \infty}\sum_{\{j_1, \dots , j_m \} = [1,\dots, n]}\int_{E}\prod_{i = 1}^m\frac{\indicate\{X_{j_i}(\omega) \in A_i\}}{n}\;\PR(d\omega) \label{eqn:DF3}\displaybreak[0]\\
     &= \lim_{n \to \infty}\frac{n(n-1)\dots (n-m+1)}{n^m}\;\PR(E, X_1 \in A_1, \dots , X_m \in A_m) \label{eqn:DF4}\displaybreak[0]\\
     &= \PR(E, X_1 \in A_1, \dots , X_m \in A_m) \notag \displaybreak[0]\\
     &= \int_E \indicate\{X(\omega) \in B\}\;\PR(d\omega). \notag
\end{align}
\cref{eqn:DF1} follows because \cref{eqn:DFdef} holds almost surely and since the union of $m$ measure zero sets has measure zero. \cref{eqn:DF2} follows by linearity and the Dominated Convergence Theorem (all terms are bounded by $1$). \cref{eqn:DF3} follows from 
\begin{align*}
    0 &\leq \lim_{n \to \infty}\sum_{\{j_1, \dots, j_m\} \neq [1,\dots, n]}\int_{E}\prod_{i = 1}^m\frac{\indicate\{X_{j_i}(\omega) \in A_i\}}{n}\;\PR(d\omega) \\
    &\leq \lim_{n \to \infty}\sum_{\{j_1, \dots, j_m\} \neq [1,\dots, n]} \frac{1}{n^m} \\
    &= \lim_{n \to \infty} \frac{n^m-n(n-1) \dots (n-m+1)}{n^m} \\
    &= 0.
\end{align*}
\cref{eqn:DF4} follows from exchangeability and since $E \in \mathcal{G} \subseteq \mathcal{E}_{\infty}$. Therefore $(X_n)_{n \in \N}$ is conditionally iid given $\mathcal{G}$. \\

Next, we show that \cref{item:DFcond2} implies \cref{item:DFcond1}. Suppose that $(X_n)_{n \in \N}$ is conditionally iid given some sub $\sigma$-algebra $\mathcal{G} \subseteq \Sigma$. Let $\kappa$ be a regular conditional distribution of $X_1$ given $\mathcal{G}$. In particular, $\kappa$ is a Markov kernel from $(\Omega, \mathcal{G})$ to $(S, \mathcal{S})$. Fix $A \in \mathcal{S}$. Then $(\indicate\{X_n \in A\})_{n \in \N}$ are conditionally iid given $\mathcal{G}$. By \cref{prop:condiidslln}, 
$$\frac{\indicate\{X_1(\omega) \in A\} + \dots + \indicate\{X_n(\omega) \in A\}}{n} \to \Ex[\indicate\{X_1 \in A\}|\mathcal{G}](\omega) = \kappa(\omega, A) \quad \text{a.s.}$$
\cref{prop:exchslln} implies that 
$$\frac{\indicate\{X_1(\omega) \in A\} + \dots + \indicate\{X_n(\omega) \in A\}}{n} \to \Ex[\indicate\{X_1 \in A\}|\Sigma_{\infty}](\omega) \quad \text{a.s.}$$
Therefore $\PR$-almost surely 
$$\kappa(\omega, A) = \Ex[\indicate\{X_1 \in A\}|\Sigma_{\infty}](\omega).$$
Since this holds for all $A \in \mathcal{S}$, $\kappa$ is a regular conditional distribution of $X_1$ given $\Sigma_{\infty}$.
\end{proof}
\begin{remark}
    The sub $\sigma$-algebra $\mathcal{G}$ in \cref{item:DFcond2}  above can be chosen to be either of $\mathcal{I}_{\infty}, \mathcal{T}_{\infty}, \mathcal{E}_{\infty}$.
\end{remark}

\section{Marginally Radon processes}

In this section we prove our sufficient condition for a process to be conditionally iid, \cref{t.sufficient}.
%
%

%
%
Let us introduce some notation which we will be using in the theorems and proofs throughout this section.
Given a sequence $X = (X_n)_{n \in \N}$ of random elements 
$(\Omega,\Sigma) \to (S,\mathcal S)$, where $(\Omega,\Sigma, \PR)$ is the basic probability space,
let us use the following short-hand:
For $\omega \in \Omega$, $n\in\N$, and $B\in\mathcal S$, write
\[
\mu_{\omega,n}(B) := \frac1n \sum_{i<n} \indicate(X_i(\omega) \in B). 
\]

We write $\prob(\Omega)$ for the set of probability measures on the measurable space $(\Omega, \Sigma)$
and $\radonprob(S)$ for the set of Radon propability measures on $S$, where $S$ is a topological space.
We write $\Borel(S)$ for the Borel $\sigma$-algebra on $S$.

When we speak of a net $(\mu_\lambda)_{\lambda \in \Lambda}$, it is understood that $\Lambda$ is a directed set whose ordering will either go unmentioned or can be inferred from context (namely, 
we will only use $\Lambda = \Nats$ except in the upcoming definition of the A-topology and in the appendix).

The \inlinedfn{A-topology} on $\prob(S)$ is defined as follows:
A net $(\mu_\lambda)_{\lambda \in \Lambda}$ converges to $\mu$
if and only if 
for every closed $F \subseteq S$, 
\[
\limsup_\lambda \mu_\lambda(F) \leq \mu(F).
\]

We shall use the following theorem about convergence of measures:
\begin{theorem}\label{t.topsoe.radon}
Suppose $S$ is a Hausdorff topological space and
$\{\mu_n \setdef n\in\N\} \subseteq \radonprob(S)$.
Further suppose that
$\{\mu_n \setdef n\in\N\}$ is tight.
Then every sequence from $\{\mu_n \setdef n\in\N\}$
has a subsequence which converges in $\radonprob(S)$ with respect to the A-topology.
\end{theorem}

This theorem can be seen, for instance, as a corollary to a very general, measure theoretic theorem of
Topsøe (see \cite{topsoe}; be aware that this author refers to what is here called A-topology as the \emph{w-topology}).
Due to the technical nature of the proof, we relegate it to the appendix (see \cref{appendix}).

A crucial steps in the proof of \cref{t.sufficient} is encapsulated in the technical \cref{l.uniform}.
But first, let us prove a preparatory lemma.
\begin{lemma}\label{l.bound}
Suppose $X$ is a random element in the measurable space $(S,\mathcal S)$, defined on the probability space $(\Omega,\Sigma,\PR)$.
Suppose further that 
$B \in \mathcal S$, $\mathcal G$ is a sub-$\sigma$-algebra of $\Sigma$ and $\omega \mapsto Y(\omega)$ is a version of 
$\Ex\big(\indicate(X \in B) \vert \mathcal G\big)$.
Then 
\[
\PR(X \in B) \leq \varepsilon^2 \Rightarrow \PR\big(\{\omega\in\Omega \setdef Y(\omega) < \varepsilon  \}\big) \geq 1-\varepsilon
\]
\end{lemma}
\begin{proof}
Let $B \in \mathcal S$ with $\PR(X \in B) \leq \varepsilon^2$ be given.
Writing
\[
F = \{\omega\in\Omega \setdef Y(\omega) \geq \varepsilon\}
\]
suppose towards a contradiction that
$\PR(\Omega\setminus F) < 1 -\varepsilon$, or equivalently, $\PR(F) > \varepsilon$.
Then
\[
\varepsilon^2 \geq \PR(X \in B) = \int Y(\omega) \; \mathrm d \PR(\omega)
\geq \int_{F} \varepsilon \; \mathrm d \PR(\omega) = \varepsilon \PR(F) > \varepsilon^2,
\]
a contradiction.
\end{proof}
This brings us to the crucial lemma.
\begin{lemma}\label{l.uniform}
Suppose $X = (X_n)_{n\in\N}$ is an exchangeable process on the probability space $(\Omega,\Sigma,\PR)$
taking values in the measurable space $(S,\mathcal S)$.
Suppose further that $(B_m)_{m\in\N}$ is a sequence from $\mathcal S$ such that
$\PR(X_1 \in B_m) \to 0$ as $m\to \infty$.
Then
\[
(\forall^\PR \omega \in \Omega)(\forall \varepsilon \in \NNReals) (\exists m \in \N) (\forall n\in\N) \; \mu_{\omega,n}(B_m) < \varepsilon.
\]
\end{lemma}
\begin{proof}
Fix an arbitrary decreasing sequence $(\varepsilon_m)_{m\in\N}$ of positive real numbers converging to zero.
By passing from $(B_m)_{m \in \Nats}$ to a subsequence if necessary, we may assume
that for each $m \in \N$,
\[
\PR(X_1 \in B_m) < {\varepsilon_m}^2.
\]
We may moreover assume that $(B_m)_{m\in\N}$ is $\subseteq$-decreasing.
Write
\[
F_m = \{\omega\in\Omega \setdef \lim_n \mu_{\omega,n}(B_m) < \varepsilon_m\}
\]
noting that the limit exists for $\PR$-almost all $\omega$.
Since $\omega \mapsto \lim_n \omega_{\omega,n}(B_m)$ is a version of
$\Ex\big(\indicate(X \in B_m) \vert \mathcal{E}_\infty\big)$ by the strong law of large numbers for exchangeable sequences (\cref{prop:exchslln})
we know by the previous lemma (\cref{l.bound})
that
\[
\PR(F_m) \geq 1 - \varepsilon_m
\]
and therefore,
\begin{equation}
\PR\left(\bigcap_{m \in \nat} \bigcup_{m' \geq m} F_{m'}\right) = 1.
\end{equation}
This means that for $\PR$-amost all $\omega \in \Omega$,
\begin{equation}\label{e.limit}
(\forall m \in \N) (\exists m' \geq m) \;  \lim_n \mu_{\omega,n}(B_{m'}) < \varepsilon_{m'}.
\end{equation}
Likewise for each $n\in\N$, a similar argument shows that for $\PR$-amost all $\omega \in \Omega$,
\begin{equation}\label{e.n}
(\forall m \in \N) (\exists m' \geq m) \;  \mu_{\omega,n}(B_{m'}) < \varepsilon_{m'}.
\end{equation}

Let us write $F$ for a set of $\omega \in \Omega$ satisfying \eqref{e.limit} and \eqref{e.n} for each $n\in\N$, 
and such that $\PR(F)=1$.

\medskip

To prove the lemma, it now suffices to prove the following:
\begin{claim}
It holds that
\begin{equation}\label{e.uniform}
(\forall  \omega \in F)(\forall \varepsilon \in \NNReals) (\exists m \in \N) (\forall n\in\N) \; \mu_{\omega,n}(B_m) < \varepsilon
\end{equation}
\end{claim}
To see this, let $\omega \in F$ and $\varepsilon > 0$ be given. Find $m$ such that
\[
\varepsilon_{m} < \frac\varepsilon 2.
\]
By \eqref{e.limit} we can find $m' \geq m$
such that
\[
\lim_n \mu_{\omega,n}(B_{m'}) < \varepsilon_{m'}.
\]
Since $(\varepsilon_m)_{m\in\N}$ is decreasing, this implies
\[
\lim_n \mu_{\omega,n}(B_{m'}) < \frac\varepsilon 2.
\]
We can thus find $\bar n$ such that for all $n \geq \bar n$,
\begin{equation}\label{e.tail}
\mu_{\omega, n}(B_{m'}) < \varepsilon.
\end{equation}
Applying, repeatedly, a similar argument for each $n < \bar n$, we can find $m(\bar n-1) \geq  \hdots \geq m(1) \geq m'$ such 
that for each $n <\bar n$, 
\[
 \mu_{\omega,n}(B_{m(n)}) < \varepsilon_{m(n)}.
\]
Since $(B_m)_{m\in\N}$ is $\subseteq$-decreasing and $(\varepsilon_m)_{m\in\N}$ is also decreasing,
for every $n < \bar n$,
\[
 \mu_{\omega,n}(B_{m(\bar n-1)}) \leq  \mu_{\omega,n}(B_{m(n)})  < \varepsilon_{m(n)} <  \varepsilon .
\]
Likewise, since $B_{m(\bar n-1)} \subseteq B_{m'}$ and by \eqref{e.tail}, 
for $n \geq \bar n$,
\[
\mu_{\omega, n}(B_{m(\bar n-1)}) < \varepsilon
\] 
showing that $m = m(\bar n-1)$ satisfies the claim, and thus, finishing the proof of the lemma.
\end{proof}

We now have all the necessary prerequisites for a proof of this section's main theorem.
\begin{proof}[Proof of \cref{t.sufficient}]
Let $X$ be as in the theorem.
Let us write $(\Omega, \Sigma, \PR)$ for the basic probability space, and let us suppose that the process $X$ takes values in 
$S$ carrying a Hausdorff topology; we consider $S$ as a measurable space equipped with $\mathcal B = \Borel(S)$.

By \cref{thm:condiiddefinetti}, it is enough to produce a 
regular conditional distribution of $X_1$, given $\mathcal E_\infty$.
To this end we will apply \cref{t.topsoe.radon} to a large subset of the family of sequences $\{(\mu_{\omega,n})_{n\in\N} \setdef \omega \in\Omega\}$.
Clearly, for every $\omega$ and $n$, $\mu_{\omega,n}$ is a (finitely supported) measure on $(S, \mathcal B)$,
and moreover, $\mu_{\omega,n} \in \radonprob(S)$, 
which is a necessary condition in \cref{t.topsoe.radon}.

Fix a sequence $(K_m)_{m\in\N}$ of compact subsets of $S$ with $\PR(X_1 \in K_m) \to 1$ for $m \to \infty$.
By \cref{l.uniform} we can find a set $F \in \Sigma$ with $\PR(F) = 1$ such that  
\eqref{e.uniform} holds with $B_m = S \setminus K_m$.

\begin{claim}
For each $\omega \in F$, the family $\{\mu_{\omega,n} \setdef n \in \N\}$ is tight.
\end{claim}
\begin{proof}
Given $\omega \in F$, $\varepsilon > 0$, by \eqref{e.uniform} with $B_m = S \setminus K_m$, we can find an $m \in \N$ such that 
$\mu_{\omega,n}(K_m) > 1-\varepsilon$ for all $n\in\N$.
\end{proof}

Momentarily fix $\omega \in F$.
By the previous claim and by \cref{t.topsoe.radon}, 
we can find a subsequence $\big(\mu'_{\omega,n}\big)_{n\in\N}$ of $(\mu_{\omega,n})_{n\in\N}$
and a Radon probability measure $\mu_\omega$ on $S$ such that
$\mu'_{\omega,n} \longrightarrow \mu_\omega$ in the A-topology as $n\to\infty$.

\begin{claim}
For every $B \in \mathcal B$, for $\PR$-almost all $\omega \in \Omega$,
\begin{equation}\label{e.limit.B}
\lim_n \mu_{\omega,n}(B) = \mu_\omega(B).
\end{equation}
\end{claim}
\begin{proof}
Let $B \in \mathcal B$ be given.
Since $\distribution(X_1)$ is Radon, we can find %
sequences $(K^j_m)_{m\in\N}$ for each $j \in \{0,1\}$, where each $K^j_m$ is a compact subset of $S$ such that
\begin{itemize}
\item $K^1_m \subseteq B$, $\PR(X_1 \in B\setminus K^1_m) \to 0$ as $m \to \infty$, 
\item $K^0_m \subseteq S\setminus B$, 
and 
 $\PR\big(X_1 \in (S\setminus B)\setminus K^0_m\big) \to 0$ as $m \to \infty$.
 \end{itemize}
 
Fix $F \in \Sigma$ with $\PR(F) =1$ such that for each $\omega \in F$, 
 \eqref{e.uniform} holds with respect to the sequence given by $B_m = B \setminus K^1_m$, as well as with the sequence given by $B_m = (S \setminus B) \setminus K^0_m$, and finally, such that 
 $\big(\mu_{\omega,n}(K^j_m)\big)_{n\in\N}$ converges in $[0,1]$ for each $\omega \in F$ and each $m \in \N$ and $j \in\{0,1\}$.

Towards a contradiction, suppose that $\omega \in F$ and \eqref{e.limit.B} fails.
Assume first that the left-hand side in \eqref{e.limit.B} is strictly above the right-hand side; fix $\varepsilon > 0$ such that
\begin{equation}\label{e.limit.fail}
\lim_n \mu_{\omega,n}(B)  > \mu_\omega(B) + \varepsilon.
\end{equation}
Since $\omega \in F$ we can find $m \in \N$ such that $\mu_{\omega,n}(B \setminus K^1_m) < \frac\varepsilon 2$,
and hence $\mu_{\omega,n}(B) < \mu_{\omega,n}(K^1_m) + \frac\varepsilon 2$, for all $n\in\N$.
We now have
\begin{equation}\label{e.limit.contradiction}
\lim_n \mu_{\omega,n}(B) \leq \lim_n \mu_{\omega,n}(K^1_m) + \frac\varepsilon 2 \leq \mu_\omega(K^1_m) + \frac\varepsilon 2
\leq  \mu_\omega(B) + \frac\varepsilon 2
\end{equation}
where the second inequality uses convergence in the A-topology, and that since $\omega \in F$, $\big(\mu_{\omega,n}(K^j_m)\big)_{n\in\N}$ is convergent.
But \eqref{e.limit.contradiction} contradicts \eqref{e.limit.fail} above.

In case that the left-hand side in \eqref{e.limit.B} is strictly below the right-hand side,
replace $B$ with $S\setminus B$ and $(K^1_m)_{m\in\N}$ with $(K^0_m)_{m\in\N}$ in the above argument.
Having reached a contradiction in either case, we infer that \eqref{e.limit.B} holds for each $\omega \in F$.
\end{proof}

We have thus constructed $\omega \mapsto \mu_\omega$ such that 
for any $B \in \mathcal B$,
for $\PR$-almost every $\omega \in \Omega$ it holds that
\[
\mu_\omega(B)
= \lim_{n \to\infty} \mu_{\omega,n}(B).
\]
By strong law of large numbers for exchangeable sequences (\cref{prop:exchslln}),
 therefore, $\mu_\omega(B) = \Ex[\indicate(X_1\in B) \vert \mathcal E_\infty]$ for $\PR$-almost every $\omega \in \Omega$.
In particular, $\omega \mapsto \mu_\omega(B)$ is measurable $(\Omega,\Sigma) \to [0,1]$ and hence, $(\mu, B) \mapsto \mu_\omega(B)$ is a regular conditional distribution of $X_1$ given $\mathcal E_\infty$.
\end{proof}

\section{When is a measure Radon?}\label{s.when.radon}

Following \citet{alam}, let us call a probability measure $\mu$ on a topological space $S$ \textbf{tight} if for every $\varepsilon \in \PosReals$ there is a compact set $K \subseteq S$ such that $\mu(K) > 1- \varepsilon$.
Suppose $K \subseteq S$; recall 
that $\mu$ is said to be \textbf{outer regular on $K$} if 
\[
\mu(K) = \inf\{\mu(O) \setdef K \subseteq O, \text{ $O\subseteq S$ is open}\}.
\]
We now show:

\begin{proposition}\label{p.when.radon}
Suppose that $\mu$ is a Borel probability measure on a Hausdorff space $S$.
Then $\mu$ is Radon if and only if 
$\mu$ is both tight and 
outer regular on compact sets.
\end{proposition}
Note again that if $S$ is a metric space, it is known that tightness of $\mu$ alone already implies that $\mu$ is Radon.

For the proof, also recall that given $G \subset S$,
 $\mu$ is said to be \textbf{inner regular on $G$ with respect to $\mathcal K$} if 
\[
\mu(G) = \inf\{\mu(K) \setdef K \subseteq G, \text{ $K\in\mathcal K$}\}.
\]
Moreover, \textbf{inner regular on $G$} means inner regular with respect to the collection of compact sets.

\begin{proof}
We only have to show the direction from tightness and outer regularity on compact sets to Radon-ness, the other being well-known.

First we show that $\mu$ is outer regular on closed sets:
Given a closed set $F\subseteq S$ and $\varepsilon\in\PosReals$ , 
find $K$ compact such that $\mu(X\setminus K) < \frac\varepsilon2$, 
then find $O$ such that $F\cap K \subseteq O$ and $\mu\big(O \setminus (K\cap F)\big) < \frac\varepsilon2$.
Then $F \subseteq O \cup (X \setminus K)$ and %
\begin{align*}
\mu\big(O \cup (X\setminus K)\big)
=
 \mu(O \cap K\cap F) &+ \mu\big(O\cap K \setminus F\big) +\mu(X\setminus K) 
 \\
 \leq 
 \mu(F) &+ \mu\big(O \setminus (K\cap F)\big) +\mu(X\setminus K) 
 <
 \mu(F) + \varepsilon.
 \end{align*}
Let us write $\mathcal B$ for the collection of sets on which $\mu$ is both inner regular with respect to the closed sets and outer regular.
We have just shown that $\mathcal B$ contains the closed subsets of $S$.
It follows by standard arguments that $\mathcal B \supseteq \Borel(S)$.

Finally, $\mu$ is inner regular on the closed sets:
Let $F \subseteq S$ be closed and let $\varepsilon \in \PosReals$ be given. 
Find $K$ such that $\mu(X\setminus K) < \varepsilon$. 
Then in particular, $\mu(F) = \mu(K\cap F) + \mu(F\setminus K) < \mu(K \cap F) + \varepsilon$. 
Since $\mu$ is inner regular on Borel sets with respect to the closed sets, and also inner regular on closed sets, $\mu$ is inner regular and hence Radon.
\end{proof}

\section*{Acknowledgments}

DMR is supported in part by Canada CIFAR AI Chair funding through the Vector Institute and an NSERC Discovery Grant.

\appendix

\section{Further proof details}\label{appendix}

We give the proof of one of the main ingredients for \cref{t.sufficient}, namely the theorem about convergence of measures, 
\cref{t.topsoe.radon}.
We repeat \cref{t.topsoe.radon} here for the readers convenience:
\begin{theorem}\label{a.t.topsoe.radon}
Suppose $S$ is a Hausdorff topological space and
$\{\mu_n \setdef n\in\N\} \subseteq \radonprob(S)$.
Further suppose that
$\{\mu_n \setdef n\in\N\}$ is tight.
Then every sequence from $\{\mu_n \setdef n\in\N\}$
has a subsequence which converges in $\radonprob(S)$ with respect to the A-topology.
\end{theorem}

As we have mentioned, this can be derived as a corollary to the following very general theorem of
Topsøe (see \cite{topsoe}; compare also \cite{topsoe-book} and \cite{topsoe-1974}).
Due to its generality, we will need to set up a bit of nomenclature, taken from \cite{topsoe}, before we can state it.

\medskip

In the next theorem, 
we will be speaking about a set $S$ and two families $\mathcal K, \mathcal G \subseteq \powerset(S)$ (sometimes such families are called ``pavings'').
In this situation, let us write $\mathcal B(\mathcal K)$ for the $\sigma$-field generated by the set of
$E \subseteq S$ such that $E \cap K \in \mathcal K$ for every $K\in \mathcal K$.
We say that $\mathcal G' \subseteq \mathcal G$ \inlinedfn{dominates} $\mathcal K$ to mean that
\[
(\forall K \in \mathcal K) (\exists G \in \mathcal G')\; K \subseteq G.
\]
Finally, let us write $\measreg(S; \mathcal K)$ 
for the set of measures $\mu$ with $\dom(\mu) = \mathcal B(\mathcal K)$
which are finite (that is, $\mu(S) = 1$)  
and regular with respect to $\mathcal K$, that is, 
$\mu(B) = \sup\{\mu(K) \setdef K \in \mathcal K, K \subseteq B\}$ for all $B \in \mathcal B(\mathcal K)$.

\begin{theorem}
[Topsøe]
\label{t.topsoe}
Suppose we have a set $S$ and $\mathcal K, \mathcal G \subseteq \powerset(S)$ such that
\begin{enumerate}[label=\Roman*.]
\item $\emptyset \in \mathcal K$, and $\mathcal K$ is closed under finite unions and countable intersections  (i.e.,
``$\mathcal K$ is a $(\emptyset, \bigcup \mathrm{f}, \bigcap \mathrm{c})$-paving''),
\item $\emptyset \in \mathcal G$, and $\mathcal G$ is closed under finite unions and finite intersections  (i.e.,
``$\mathcal G$ is a $(\emptyset, \bigcup \mathrm{f}, \bigcap \mathrm{f})$-paving''),
\item $(\forall K \in \mathcal K) (\forall G \in \mathcal G) \; K \setminus G \in \mathcal K$,
\item $\mathcal G$ seperates sets in $\mathcal K$, that is, for $K_0, K_1 \in \mathcal K$ there are $G_i \in \mathcal G$ with 
$K_i \subseteq G_i$ for each $i \in\{0,1\}$ such that $G_0 \cap G_1 = \emptyset$,
\item $\mathcal K$ is semi-compact, that is, every countable subset of $\mathcal K$ with the finite intersection property has non-empty intersection.
\end{enumerate}
Let $\bar \mu = (\mu_\lambda)_{\lambda \in \Lambda}$ be a net from $\measreg(S; \mathcal K)$. Then the following are equivalent:
\begin{enumerate}[label=(\Alph*)]
\item Every subnet of $\bar \mu$ has a further subnet which converges in the A-topology 
(``$\bar \mu$ is compact'').
\item The conjunction of the following:
\begin{enumerate}[label=(\roman*)]
\item $\limsup_\lambda \mu_\lambda(S) < \infty$, and
\item\label{i.topsoe} For every $\mathcal G' \subseteq \mathcal G$ which dominates $\mathcal K$,
\begin{equation}\label{e.topsoe}
\inf_{\mathcal G''} \limsup_\lambda \min_{G \in \mathcal G''} \mu_\lambda(S\setminus G) = 0
\end{equation}
where the infimum is taken over all finite $\mathcal G'' \subseteq \mathcal G'$.
\end{enumerate}
\end{enumerate}

\end{theorem}

\begin{proof}[Proof of \cref{a.t.topsoe.radon} from \cref{t.topsoe}]
Let a Hausdorff topological space $S$ and a tight family of Radon probabilities $\{\mu_n \setdef n\in\Nats\}$ be given.
Let, for each $n \in \Nats$,
\[
\mathcal N_n = \{N \subseteq S \setdef \big(\exists N' \in \Borel(S)\big) N \subseteq N' \text{ and } \mu_n(N') =0\big\}
\]
and
\[
\mathcal B = \left\{B \Delta N \;\mid\;  N \in \bigcap_{n\in\Nats} \mathcal{N}_n \right\}
\]
Let $\hat\mu_n$ be the obvious extension of $\mu_n$ to $\mathcal B$.
Take $\mathcal G$ to be the open subsets of $S$, $\mathcal K$ to be the compact subsets of $S$.
Then by Radon-ness and tightness, $\mathcal B(\mathcal K) \subseteq \mathcal B$ and $\hat{\mu}_n \res \mathcal B(\mathcal K) \in \measreg(S; \mathcal K)$. 
Trivially, \eqref{e.topsoe} is equivalent to the following:
\[
(\forall\varepsilon \in \NNReals) \left(\exists \mathcal G'' \in [\mathcal G']^{<\Nats}\right)(\exists \lambda_0 \in \Lambda)
(\forall \lambda \mathbin{>_\Lambda} \lambda_0)(\exists G \in \mathcal G'') \; \mu_\lambda(S \setminus G) < \varepsilon
\]
and therefore, clearly, tightness implies that \cref{i.topsoe} holds of the sequence $n \mapsto \hat{\mu}_n \res \mathcal B(\mathcal K)$.
It follows that every subnet of $(\mu_n)_{n\in\Nats}$ has a further subnet which converges in the A-topology.
\end{proof}

\section{Definitions}

\begin{definition}[Markov Kernel]
    Let $(\Omega_1, \Sigma_1), (\Omega_2, \Sigma_2)$ be measurable spaces. A map $\kappa : \Omega_1 \times \Sigma_2 \to [0,1]$ is called a $\inlinedfn{Markov kernel}$ if
    \begin{enumerate}
        \item $\omega_1 \mapsto \kappa(\omega_1, A_2)$ is $\Sigma_1$-measurable for any $A_2 \in \Sigma_2$,
        \item $A_2 \mapsto \kappa(\omega_1, A_2)$ is a probability measure on $(\Omega_2, \Sigma_2)$ for any $\omega_1 \in \Omega_1$. \label{item:defkernel}
    \end{enumerate}
    We also use the notation $\kappa_{\omega_1}(A_2) := \kappa(\omega_1, A_2)$.   
\end{definition}

\begin{definition}[Regular Conditional Distribution]
    Let $X$ be a random variable from the probability space $(\Omega, \Sigma, \PR)$ to the measurable space $(S, \mathcal{S})$ and let $\mathcal{G} \subseteq \Sigma$ be a sub $\sigma$-algebra. A Markov kernel $\kappa$ from $(\Omega, \mathcal{G})$ to $(S, \mathcal{S})$ is called a \inlinedfn{regular conditional distribution} of $X$ given $\mathcal{G}$ if for all $B \in \mathcal{S}$
    $$\kappa(\omega, B) = \Ex[\indicate\{X \in B\}|\mathcal{G}](\omega)$$
    for $\PR$-almost all $\omega \in \Omega$.
\end{definition}

\begin{definition}[Exchangeable, Left-Shift Invariant, Tail $\sigma$-algebras]
    Let $(X_n)_{n \in \N}$ be a sequence of measurable maps from $(\Omega, \Sigma)$ to $(S, \mathcal{S})$. Define $X : (\Omega, \Sigma) \to (S^{\N}, \mathcal{S}^{\N})$ to have coordinate projections given by the $(X_n)_{n \in \N}$. Let $\tau : S^{\N} \to S^{\N}$ be the left-shift operator defined by $\tau(s_1, s_2, \hdots) := (s_2, s_3, \hdots)$. The \inlinedfn{left-shift invariant $\sigma$-algebra} on $\Omega$ is
    \begin{equation*}
        \mathcal{I}_{\infty} := \{X^{-1}(A) : A \in S^{\N}, A = \tau^{-1}(A)\}.
    \end{equation*}
    The \inlinedfn{tail $\sigma$-algebra} on $\Omega$ is
    \begin{equation*}
        \mathcal{T}_{\infty} := \bigcap_{n \in \N} \sigma(X_n, X_{n+1},  \dots ). 
    \end{equation*}
    The \inlinedfn{exchangeable $\sigma$-algebra} on $\Omega$ is
    \begin{equation*}
        \mathcal{E}_{\infty} := \{X^{-1}(A) : A \in \mathcal{S}^{\N},  A = \pi ^{-1}(A) \text{ for all finite permutations $\pi$ of $\N$} \}.
    \end{equation*}
where of course,  $\pi(s_1, s_2, \hdots)$ means
$(s_{\pi(1)}, s_{\pi(2)}, \hdots )$ for $(s_1, s_2, \hdots) \in \mathcal S^\N$
\end{definition}

\renewcommand{\bibfont}{\small}
\printbibliography

\end{document}

%% file: headers.tex
\pdfoutput=1
\usepackage{amsmath,amsthm,wrapfig}
	\usepackage[utf8]{inputenc} 
\usepackage{amsmath, amssymb,bm, cases, mathtools, thmtools}
\usepackage{verbatim}
\usepackage{graphicx}\graphicspath{{figures/}}
\usepackage{multicol}
\usepackage{tabularx}
\usepackage[usenames,dvipsnames]{xcolor}
\usepackage{mathrsfs} 
\usepackage{caption}
\usepackage{algorithm}
\usepackage{algorithmicx}
\usepackage[noend]{algpseudocode}
\usepackage{array,booktabs,arydshln,xcolor}

\usepackage[%
		minnames=1,maxnames=99,maxcitenames=2,
		style=numeric-comp,%
		sorting=none,
		sortcites=false, %
		doi=false,url=false,
		uniquename=init,
		giveninits=true,
		hyperref,natbib,
		backend=biber]{biblatex}
\renewbibmacro{in:}{%
	\ifentrytype{article}{}{\printtext{\bibstring{in}\intitlepunct}}}

\usepackage[T1]{fontenc}
\usepackage{inconsolata}
\usepackage{dsfont}

\usepackage{listings} %

\usepackage{enumitem}
\usepackage{booktabs}       %
\usepackage{nicefrac}       %

\usepackage{lineno}
\usepackage{bbm}

\usepackage{caption}
\usepackage{subcaption}

\usepackage{datetime}

\DeclareMathAlphabet\EuRoman{U}{eur}{m}{n}
\SetMathAlphabet\EuRoman{bold}{U}{eur}{b}{n}

\declaretheorem[style=plain,numberwithin=section,name=Theorem]{theorem}
\declaretheorem[style=plain,sibling=theorem,name=Lemma]{lemma}
\declaretheorem[style=plain,sibling=theorem,name=Proposition]{proposition}

\declaretheorem[style=definition,sibling=theorem,name=Definition]{definition}

\declaretheorem[style=remark,qed=$\triangleleft$,sibling=theorem,name=Remark]{remark}
\numberwithin{theorem}{section}

\usepackage{xparse}
\usepackage{xstring}
\usepackage{xspace}

%% file: defs.tex
\def\[#1\]{\begin{align}#1\end{align}}
\def\*[#1\]{\begin{align*}#1\end{align*}}

\newcommand\optparen[1]{\ifthenelse{\equal{#1}{}}{}{(#1)}}

\newcommand{\Reals}{\mathbb{R}}

\newcommand{\Nats}{\mathbb{N}}

\newcommand{\NNReals}{\Reals_+}

\newcommand{\dom}{\textrm{dom}}

\DeclareMathOperator*{\newlim}{\mathrm{lim}\vphantom{\mathrm{infsup}}}
\DeclareMathOperator*{\newmin}{\mathrm{min}\vphantom{\mathrm{infsup}}}

\DeclareMathOperator*{\newinf}{\mathrm{inf}\vphantom{\mathrm{infsup}}}
\DeclareMathOperator*{\newsup}{\mathrm{sup}\vphantom{\mathrm{infsup}}}
\renewcommand{\lim}{\newlim}
\renewcommand{\min}{\newmin}

\renewcommand{\inf}{\newinf}
\renewcommand{\sup}{\newsup}

\newcommand{\PosReals}{\Reals_{+}}

\newcommand{\e}{\mathrm{e}}

\newcommand{\lcrx}[4][{-1}]{
	\IfEq{#1}{-1}{\left #2 {{{{#3}}}} \right #4}{
   	\IfEq{#1}{0}{#2 {{{{#3}}}} #4}{
	\IfEq{#1}{1}{\bigl #2 {{{{#3}}}} \bigr #4}{
	\IfEq{#1}{2}{\Bigl #2 {{{{#3}}}} \Bigr #4}{
	\IfEq{#1}{3}{\biggl #2 {{{{#3}}}} \biggr #4}{
	\IfEq{#1}{4}{\Biggl #2 {{{{#3}}}} \Biggr #4}{
    \GenericWarning{"4th argument to lcrx must be -1, 0, 1, 2, 3, or 4"}
    }}}}}}}

\newcommand{\Borel}{\operatorname{Borel}}

\newcommand{\inlinedfn}{\textbf}

%% file: preamble.tex
\theoremstyle{plain}
\newtheorem{claim}[theorem]{Claim}

\newtheorem*{claim*}{Claim}
\newtheorem*{theorem*}{Theorem}
\newtheorem*{lemma*}{Lemma}
\newtheorem*{proposition*}{Proposition}
\newtheorem*{corollary*}{Corollary}
\newtheorem*{subclaim*}{Subclaim}

\theoremstyle{definition}

\newtheorem*{definition*}{Definition}
\newtheorem*{example*}{Example}
\newtheorem*{fact*}{Fact} 

\theoremstyle{remark}
{\end{minipage}\end{equation}}

{\end{minipage}\end{equation}}

\newenvironment{eqpar*}{\begin{equation*}\begin{minipage}{0.8\columnwidth}}%
{\end{minipage}\end{equation*}}

\DeclareMathOperator{\prob}{\mathcal M_1}

\newcommand{\bsim}[1]{\mathbin{\sim_{#1}}}
\newcommand{\nbsim}[1]{\mathbin{\not\sim_{#1}}}

\newcommand{\rsim}[1]{\bsim[r]}
\newcommand{\nrsim}[1]{\nbsim[r]}

\DeclareMathOperator{\powerset}{\mathcal{P}}

\DeclareMathOperator{\nat}{\mathbb{N}}

\providecommand{\int}{\mathbb{Z}}

\providecommand{\res}{\mathbin{\upharpoonright} }

\providecommand{\h}{\textup{ht}\;}

\providecommand{\setdef}{\;|\;}

\providecommand{\Borel}{\mathbf{Borel}}

\newcommand{\g}{g}

\providecommand{\R}{R}

\providecommand{\F}{\mathbf{F}}
\providecommand{\C}{\mathbf{C}}

\providecommand{\D}{\mathbf{D}}

\usepackage[utf8]{inputenc}
\usepackage{amssymb}
\usepackage{amsmath}
\usepackage{amsthm}
\usepackage{tikz}
\usepackage{mathtools}  %
\usepackage{mathrsfs}   %

\usepackage{graphicx}
\graphicspath{ {./images/} }

\usepackage{accents}

\newcommand{\Q}{\mathbb{Q}}
\newcommand{\N}{\mathbb{N}}
\newcommand{\Z}{\mathbb{Z}}
\newcommand{\V}{\mathbb{V}}

\newcommand{\indicate}{\mathbb{I}}
\newcommand{\I}{\mathbb{I}}

\renewcommand{\*}{{}^*}

\newcommand{\Ex}{\mathbb{E}}
\newcommand{\PR}{\mathbb{P}}

\DeclareMathOperator{\distribution}{\mathcal L}

\newcommand{\measreg}{\operatorname{\mathcal{M}_+}}

\everymath{\displaystyle}

\newtheoremstyle{named}{}{}{\itshape}{}{\bfseries}{.}{.5em}{\thmnote{#3's }#1}
\theoremstyle{named}

\theoremstyle{definition}

\DeclareMathOperator{\radonprob}{\mathcal R_1}